\documentclass[12pt,reqno]{amsart} 
\usepackage{verbatim}
\usepackage{amssymb}
\usepackage{enumerate}
\usepackage[active]{srcltx}
\numberwithin{equation}{section}

\usepackage{t1enc}
\usepackage[utf8x]{inputenc}
\newtheorem{theorem}{Theorem}[section]
\newtheorem{lemma}[theorem]{Lemma}
\newtheorem{problem}[theorem]{Problem}

\newtheorem{corollary}[theorem]{Corollary}

\theoremstyle{definition}
\newtheorem{definition}[theorem]{Definition}

\theoremstyle{remark}

\newcommand{\mc}[1]{\mathcal{#1}}
\newcommand{\mbb}[1]{\mathbb{#1}}

\newcommand{\setm}{\setminus}
\newcommand{\empt}{\emptyset}
\newcommand{\subs}{\subset}
\newcommand{\dom}{\operatorname{dom}}

\newcommand{\Fn}{\operatorname{Fn}}

\def\<{\left\langle}
\def\>{\right\rangle}

\def\br#1;#2;{\bigl[ {#1} \bigr]^ {#2} }

\newcommand{\oo}{{\omega}_1}

\DeclareMathOperator{\cf}{cf}

\DeclareMathOperator{\inte}{int}

\newcommand{\wreso}[1]{monotonically $#1$-resolvable}
\newcommand{\wresob}[1]{monotonically $#1$-resolvability}

\newcommand{\clo}[1]{\overline{\overline {#1}}}

\author[L. Soukup]{Lajos Soukup}
\address{Alfréd Rényi Institute of Mathematics, Hungarian Academy of Sciences}
\email{soukup@renyi.hu
}

\thanks
   {
   }

\author[A. Stanley]{Adrienne Stanley}
\address{Department of Mathematics, University of Northern Iowa}
\email{adrienne.stanley@uni.edu}
\thanks{The second author was 
supported by  Fulbright Scholar Program.}


\title[Resolvability in c.c.c. generic extensions]
   {Resolvability in c.c.c. generic extensions}
\date{\today}

\thanks{The preparation of this paper was supported by  OTKA grant no. K113047.}

\begin{document}
\begin{abstract}
Every crowded space $X$ is ${\omega}$-resolvable in the c.c.c generic extension 
$V^{\Fn(|X|,2})$
of the ground model.

We investigate 
what we can say about ${\lambda}$-resolvability in  c.c.c-generic extensions for 
${\lambda}>{\omega}$?

 A topological space is {\em \wreso{\oo}}
if there is a function $f:X\to \oo$ such that 
\begin{displaymath}
 \{x\in X: f(x)\ge {\alpha} \}\subs^{dense}X
\end{displaymath}
for each ${\alpha}<\oo$. 

We show that given a $T_1$ space $X$ the following statements are equivalent:
\begin{enumerate}[(1)]
 \item $X$ is  ${\omega}_1$-resolvable in some c.c.c-generic extension,
 \item $X$ is  \wreso{\oo}.
\item $X$ is  ${\omega}_1$-resolvable 
 in the Cohen-generic  extension $V^{\Fn(\oo,2)}$.
 \end{enumerate}
 
We investigate which spaces are \wreso{\oo}. 
We show that  if a topological space $X$ is c.c.c, 
and ${\omega}_1\le \Delta(X)\le |X|<{\omega}_{\omega}$,
where $\Delta(X) = \min\{ |G| : G \ne
\emptyset \mbox{ open}\}$, then $X$ is  
\wreso{\oo}.

On the other hand, 
it is also consistent, modulo the existence of a measurable cardinal,
that there is a space $Y$ with $|Y|=\Delta(Y)=\aleph_\omega$ which is not 
\wreso{\oo}.

The characterization of $\oo$-resolvability in c.c.c generic extension 
raises the following question: is it true that  crowded spaces
from the ground model are  
${\omega}$-resolvable in $V^{\Fn({\omega},2)}$? 

We show that (i) if $V=L$ then every 
crowded c.c.c. space $X$ is  ${\omega}$-resolvable in $V^{\Fn({\omega},2)}$,
(ii) if there is no weakly inaccessible cardinals, then 
every 
crowded space $X$ is  ${\omega}$-resolvable in $V^{\Fn({\omega}_1,2)}$.

On the other hand, 
it is also consistent, modulo a measurable cardinals, that there is a 
crowded space $X$  with $|X|=\Delta(X)=\oo$  such that
$X$ remains irresolvable after adding a single Cohen real.
\end{abstract}

\maketitle

\section{Introduction}

Notion of resolvability  was introduced and studied first 
 by E. Hewitt, \cite{He}, in 1943.
A topological space $X$ is {\em ${\kappa}$-resolvable} if it can be partitioned into ${\kappa}$
many dense subspaces.   $X$ is {\em resolvable} iff it is 2-resolvable, and
{\em irresolvable} otherwise.
Irresolvable spaces with many interesting 
extra properties were constructed, but   
there are no ``absolute'' examples for crowded irresolvable spaces, because 
if $X$ is a crowded space, 
then clearly
\begin{displaymath}
V^{\Fn(|X|,2)}\models \text{$X$ is ${\omega}$-resolvable.}
\end{displaymath}

In this paper we investigate 
what we can say about ${\lambda}$-resolvability in  c.c.c-generic extensions for 
${\lambda}>{\omega}$? 

\renewcommand{\thefootnote}{\fnsymbol{footnote}}

To characterize spaces which are 
${\omega}_1$-resolvable in some c.c.c-generic extension 
we introduce the notion of \wreso \kappa.
\begin{definition}
Let ${\kappa}$ be an infinite cardinal. A topological space $X$ 
is {\em \wreso {\kappa}}\footnote[2]{In \cite{TaVi} 
a ``\wreso{\omega}'' space is called ``almost-${\omega}$-resolvable''.
However, in \cite{Recent} a space $X$ is {\em almost-${\kappa}$-resolvable}
if it contains a family of ${\kappa}$ dense sets with pairwise nowhere dense intersections. 
 } if there is a function $f:X\to {\kappa}$ such that 
\begin{displaymath}
 \{x\in X: f(x)\ge {\alpha} \}\subs^{dense}X
\end{displaymath}
for each ${\alpha}<{\kappa}$. We will say that 
 $f$ {\em witnesses } that $X$ is \wreso{{\kappa}}.

\end{definition}
Clearly a space 
$X$ is \wreso{{\kappa}} iff $X$ has a partition  $\{X_{\zeta}:{\zeta}<{\kappa}\}$
of $X$ such that 
\begin{displaymath}
\inte \big(\bigcup \{X_{\zeta}:{\zeta}<{\xi}\}\big)=\empt                  
\end{displaymath}
for all ${\xi}<{\kappa}$.

\begin{theorem}\label{tm:w1-res-in-ccc}
Let $X$ be a $T_1$ topological space. The following statements are equivalent:
\begin{enumerate}[(1)]
 \item $X$ is  ${\omega}_1$-resolvable in some c.c.c-generic extension,
 \item $X$ is  \wreso{\oo},
 \item $X$ is  ${\omega}_1$-resolvable 
 in the Cohen generic extension $V^{\Fn({\omega}_1,2)}$.
\end{enumerate}
\end{theorem}

Which spaces are \wreso{\oo}?

\begin{theorem}\label{tm:ccc-spaces-below-ww}
If a topological space $X$ is c.c.c, 
and ${\omega}_1\le \Delta(X)\le |X|<{\omega}_{\omega}$, then $X$ is  
\wreso{\oo}.
\end{theorem}

\begin{theorem}\label{tm:example-of-size-measurable}
If ${\kappa}$ is a measurable cardinal, then 
 there is a space 
$X$ with $|X|=\Delta(X)={\kappa}$
which is not \wreso{{{\omega}_1}}.
\end{theorem}

What about spaces of  cardinality  ${\omega}_{\omega}$?

\begin{theorem}\label{tm:example-of-size-ww}
It is consistent, modulo the existence of a measurable cardinals, 
  that there is a 
space $X$ with $|X|=\Delta(X)={\omega}_{\omega}$
which is not \wreso{{{\omega}_1}}.
\end{theorem}

Do we really need to add $|X|$-many Cohen reals to make $X$ resolvable?

\begin{theorem}\label{tm:w-not-enough}
(1) It is consistent, modulo a measurable cardinal, that there is a 
crowded space $X$  with $|X|=\Delta(X)=\oo$ (so $X$ is \wreso{\oo})  such that
\begin{displaymath}
 V^{\Fn({\omega},2)}\models\text{``$X$ is irresolvable.''}
\end{displaymath}
(2) If $V=L$, then every crowded space with $|X|=\Delta(X)=\cf(|X|)$
is \wreso{{\omega}}, and so it is ${\omega}$-resolvable in $V^{\Fn({\omega},2)}$.\\
\noindent (3) If 
the cardinality of a crowded c.c.c space $X$  is less than the 
first  weakly inaccessible cardinal, then 
$X$ is ${\omega}$-resolvable in $V^{\Fn({\omega}_1,2)}$  
\footnote[4]{$\omega_1$ is not a misprint here}.
\end{theorem}

The almost resolvability of c.c.c spaces was investigated by Pavlov in \cite{Pa}:
on page 53 Pavlov writes that  -- mimicked Malykhin’s method by using Ulam
matrices -- he  showed that 
every crowed ccc space of cardinality
$\oo$ is almost resolvable.
In \cite[Theorem 2.22]{DoAl} a stronger result was proved:
 a crowded c.c.c. space is almost resolvable, if its cardinality 
is less than the 
first  weakly inaccessible cardinal.
Theorem \ref{tm:w-not-enough}(2) is a further improvement of 
this result because \wresob{{\omega}} implies almost resolvability.

In \cite[3.12 Problem (2)]{AIT}
the authors ask {\em if  every space with countable cellularity and cardinality less than the first
inaccessible non-countable cardinal almost-${\omega}$-resolvable?}.
As we will see Theorem \ref{tm:w-not-enough} (3) gives a positive answer to a 
weakening of this question.

\section{Characterization of $\oo$-resolvability in c.c.c extensions. }

Instead of  Theorem \ref{tm:w1-res-in-ccc}
we prove the following stronger result.

\begin{theorem}\label{tm:kappa-res-in-ccc}
Assume that  $X$ is  a crowded topological space and ${\kappa}$ is an infinite cardinal.
If ${\kappa}=\cf\big(\br {\kappa};{\omega};,\subs \big)$ then following statements are equivalent.
\begin{enumerate}[(1)]
 \item $X$ is  ${\kappa}$-resolvable in some c.c.c-generic extension,
 \item there is a function $h:X\to \br {\kappa};{\omega};$ such that 
 $\bigcup h''U={\kappa}$ for each non-empty open $U\subs X$.
 \item $X$ is  ${\kappa}$-resolvable 
 in the  Cohen-generic extension   $V^{\Fn({\kappa},2)}$.
\end{enumerate}
\end{theorem}

We say that a function $g:X\to {\kappa}$ {\em witnesses} that $X$ is 
${\kappa}$-resolvable if 
$$\{x\in X: g(x)={\alpha}\}\subs^{dense}X$$ for each ${\alpha}<{\kappa}$.

\begin{proof}
First we show that $(1) \to (2)$. 
Assume that $\mathbb{P}$ is a c.c.c. poset such that 
there is a function  $g \in V^{\mathbb{P}}$ witnessing the  
${\kappa}$-resolvability of $X$.
  
  For each $x \in X$ define 
$$h(x) = \{ \alpha < {\kappa} : \exists p^x_{\alpha} \in {\mathbb{P}} ( 
p^x_{\alpha} \Vdash \dot{g}(\check{x}) = \check{\alpha}) \} . $$ 
Since the conditions $\{p^x_{\alpha}:{\alpha}\in h(x)\}$
 are pairwise incomparable and $\mathbb P$ is c.c.c , 
 the set $h(x)$ is
countable.

We now show that the function $h$ defined above satisfies $(2)$.  Fix $\alpha < 
{\kappa}$ and $U $ an open subset of $X$.  We need to show that there exists 
$x \in U$ such that ${\alpha}\in h(x)$.    Since
$$ V^{\mathbb{P}} \models  
g^{-1}(\{ {\alpha} \} ) \subset^{dense} X$$
it follows that 
there is $x\in U$ such that 
$$ V^{\mathbb{P}} \models g(x) = {\alpha}.$$
Thus, there exists $p \in {\mathbb{P}}$ such that 
$$p \Vdash 
``  \dot{g}(\check{x})=\check{\alpha}.''$$ 
Then ${\alpha}\in h(x)$.

Next we now show that $(2) \to (3)$. 
Let $\mc A$ be a cofinal subset of $\br {\kappa};{\omega};$ with 
$|\mc A|={\kappa}$.

Let $\{A_{\alpha}:{\alpha}<{\kappa}\}$ be an enumeration of $\mc A$,
and for each $x\in X$   pick
$$h^*(x)\in \mc A \text{ such that }h^*(x)
\supset\bigcup_{{\alpha}\in h(x)} A_{\alpha}.$$

Then for all non-empty open $U$
\begin{displaymath}\tag{$+$}\label{eq:cof}
 \{h^*(x):x\in U\}\text{ is cofinal in $\br {\kappa};{\omega};$.}
\end{displaymath}

 Next we note that forcing with 
 $\Fn({\kappa}, 2)$ is the same as forcing with $\Fn({\kappa}, \omega)$.  
 Further, $\Fn({\kappa}, \omega)$ is isomorphic to 
 $$\mathbb{P} = \{ p \in \Fn(\mc A, 
 {\kappa})\ :\ \forall A \in \dom (p)\  p(A) \in  A\}.$$
Indeed, for each $A\in \mc A$ fix a bijection $\rho_A:{\omega}\to A$, and
then  for $q\in \Fn({\kappa}, \omega)$ define $\varphi(q)\in \mathbb P$
as follows:
\begin{enumerate}[(i)]
 \item $\dom(\varphi(q))=\{A_{\alpha}:{\alpha}\in \dom(q)\}$, and
\item $\varphi(q)(A_{\alpha})=\rho_{A_{\alpha}}(q({\alpha}))$ for 
$A_{\alpha}\in \dom(\varphi(q))$.
 \end{enumerate}
Then $\varphi$ is clearly an  isomorphism  between
 $\Fn({\kappa}, \omega)$ and $\mathbb{P}$.
 
 We will proceed using $\mathbb{P}$.

Let $G$ be a $\mathbb{P}$-generic filter, and   let $g = \bigcup G$.  Then $g\in 
V^{\mathbb{P}}$ and $g:\mc A \to {\kappa}$ such that $g(A)\in A$.  

We claim that $f=g\circ h^*$
 witnesses that $X$ is ${\kappa}$-resolvable.

 Fix $\alpha < {\kappa}$ and an open $U \subset X$.

 Let $q \in \mathbb{P}$ be arbitrary. Then,
by \eqref{eq:cof}, there is $x\in U$ such that 
$$\{\alpha\}\cup \bigcup dom(q)\subsetneq h^*(x).$$
Then $h^*(x)\notin \dom(q)$, and ${\alpha}\in h^*(x)$,
so 
\begin{displaymath}
 p=q\cup\{\<h^*(x),{\alpha}\>\}\in \mbb P_1,
\end{displaymath}
and 
\begin{displaymath}
 p\Vdash (g\circ h^*)(\check x)=\check{\alpha}.
\end{displaymath}

Thus, by genericity, there is $p \in G$ and $x\in U$  such that 
\begin{displaymath}
 p \Vdash ( 
\dot{g\circ h^*}(\check {x})=\check{\alpha}).
\end{displaymath}
Hence
 $$V^{\mathbb{P}}\models X \, \mbox{is} \, 
{{\kappa}}\mbox{-resolvable}.$$
\medskip
Finally $(3)\to (1)  $  is trivial.
\end{proof}

\begin{problem}
Can we drop the assumption ${\kappa}=\cf(\br {\kappa};{\omega};,\subset)$
from Theorem \ref{tm:kappa-res-in-ccc}? 
\end{problem}

\section{On \wresob{\oo} of c.c.c spaces}

We start with an easy observation.
\begin{lemma}\label{lm:open-subspace}
Let $X$ be a topological space and $\mc B\subs \mc P(X)$.
If every $B\in \mc B$ is \wreso{{\kappa}}, then so is 
$\overline{\cup\mc B}$. So every space contains a greatest \wreso{{\kappa}} subspace
(that subspace can be empty, of course).
\end{lemma}

\begin{corollary}\label{cm:dense}
Let $X$ be a topological space.  Let $Z$ be a dense subset of $X$.  If  $Z$ is \wreso{\kappa},  then $X$ is also \wreso{\kappa}.
\end{corollary}

Before proving Theorem \ref{tm:ccc-spaces-below-ww}
we prove the following ``stepping-down'' theorem.

\begin{theorem}\label{tm:stepping-down}
If  $X$ is a ${\kappa}$-c.c.,  \wreso{{\kappa}^+} space, then 
$X$ is \wreso{{\kappa}} as well.
\end{theorem}

The proof uses ideas from \cite{KuPr}.

\begin{proof}
 Since an open subspace of a ${\kappa}$-c.c.,  \wreso{{\kappa}^+} space
 is also ${\kappa}$-c.c. and   \wreso{{\kappa}^+},
by Lemma \ref{lm:open-subspace} it is enough to show that 
\begin{enumerate}
 \item[$(*)$] 
 every ${\kappa}$-c.c.,  \wreso{{\kappa}^+} space $X$ has a  
  \wreso{{\kappa}} non-empty open subset.
\end{enumerate}

 Ulam \cite{Ul} proved that there is a ``matrix'' 
 $$\<M_{{\alpha},{\zeta}}:{\alpha}<{\kappa}^+, {\zeta}<{\kappa}\>\subs \mc P({\kappa}^+)
 $$ such that 
 \begin{enumerate}[(i)]
  \item $M_{{\alpha},{\xi}}\cap M_{{\beta},{\xi}}=\empt$ for 
  $\{{\alpha},{\beta}\}\in \br {\kappa}^+;2;$ and ${\xi}\in {\kappa}$, \item  
  $M_{{\alpha},{\xi}}\cap M_{{\alpha},{\zeta}}=\empt$ for ${\alpha}\in {\kappa}^+$
  and 
  $\{{\xi},{\zeta}\}\in \br {\kappa};2;$,\medskip
  \item and
  $|M^-_{\alpha}|\le {\kappa}$, where 
  $M^-_{\alpha}={\kappa}^+\setm\bigcup_{{\zeta}<{\kappa}}M_{{\alpha},{\zeta}}$
  for  ${\alpha}<{\kappa}^+$.
 \end{enumerate}

Fix a partition $\{Y_\eta:\eta<{\kappa}^+\}$ witnessing that 
$X$ is \wreso{{\kappa}^+}.

Let $$Z_{{\alpha},{\zeta}}=\bigcup\{Y_\eta:\eta\in M_{{\alpha},{\zeta}}\}$$
for ${\alpha}<{\kappa}^+$ and ${\zeta}<{\kappa}$,
and let 
\begin{displaymath}
Z_{\alpha}=\bigcup_{{\zeta}<{\kappa}}Z_{{\alpha},{\zeta}}.         
\end{displaymath}
Since $Z_{\alpha}=\bigcup\{Y_\eta:
\eta\in {\kappa}^+\setm M^-_{\alpha}\}$,
assumption  (iii) implies that every $Z_{\alpha}$ is dense in $X$.

\noindent {\bf Case 1.}
There is ${\alpha}<{\kappa}^+$ such that 
for all ${\zeta}<{\kappa}$
\begin{displaymath}
 \bigcup_{{\zeta}\le {\xi}}Z_{{\alpha},{\xi}}\subs^{dense}Z_{\alpha}.
\end{displaymath}

Then $( Z_{{\alpha},\zeta})_{\zeta < \kappa}$ witnesses $Z_{\alpha}$ is 
\wreso{{\kappa}} and so by corollary \ref{cm:dense} , $X$ is also \wreso{{\kappa}}.

\noindent {\bf Case 2.}
For all  ${\alpha}<{\kappa}^+$ there is  ${\zeta}_{\alpha}<{\kappa}$ 
and there is an non-empty open  set $U_{\alpha}\in \tau_X$ such that 
\begin{displaymath}\tag{$\dag$}
\bigcup_{{\zeta}_{\alpha}\le {\xi}}Z_{{\alpha},{\xi}}\cap U_{\alpha}=\empt. 
\end{displaymath}

Then there is a set $I\in \br {\kappa}^+;{\kappa}^+;$ and there is 
an ordinal ${\zeta}<{\kappa}$
such that ${\zeta}_{\alpha}={\zeta}$ for all ${\alpha}\in I$.

Fix an arbitrary  $K\in \br I;{\kappa};$. By (iii) we can find  $\rho<{\kappa}^+$
such that 
\begin{displaymath}
 \bigcup_{{\alpha}\in K}M^-_{\alpha}\subs \rho.
\end{displaymath}
Let $Z=\bigcup_{\rho<\eta}Y_\eta$.
Then $Z\subs^{dense} X$ and $Z\subs Z_{\alpha}$ for all 
${\alpha}\in K$.

\medskip
\noindent{\bf Claim.}
{\em If $L\in \br K;{\kappa};$ then 
\begin{displaymath}
 \bigcap_{{\alpha}\in L}U_{\alpha}\cap Z=\empt.
\end{displaymath}}
\begin{proof}[Proof of the Claim.]
Assume on the contrary that  $z\in \bigcap_{{\alpha}\in L}U_{\alpha}\cap Z$. 
Then  $z\in Y_\eta$ for some $\rho<\eta$.

Let ${\alpha}\in L$.  
Then $\eta\in {\kappa}^+\setminus\rho\subs \bigcup_{{\xi}<{\kappa}}M_{{\alpha},{\xi}}$.
Pick ${\xi}_{\alpha}<{\kappa}$ with $\eta\in M_{{\alpha},{\xi}_{\alpha}}$.
Then $Y_\eta\subs Z_{{\alpha},{\xi}_{\alpha}}$, so 
$Z_{{\alpha},{\xi}_{\alpha}}\cap U_{\alpha}\ne \empt$, so 
${\xi}_{\alpha}<{\zeta}_{\alpha}={\zeta}$ by ($\dag$).

Since ${\zeta}<{\kappa}=|L|$, there are ${\alpha}\ne {\beta}\in \br L;2;$ such that 
${\xi}_{\alpha}= {\xi}_{\beta}$. Thus $\eta\in M_{{\alpha},{\xi}_{\alpha}}\cap 
M_{{\beta},{\xi}_{\beta}}$ which contradicts (i) because 
${\xi}_{\alpha}={\xi}_{\beta}$.
\end{proof}
\medskip

Fix an enumeration  $K=\{\chi_{\xi}:{\xi}<{\kappa}\}$,
and let $V_{{\zeta}}=\bigcup_{{\zeta}<{\xi}}U_{\chi_{\xi}}$.
Then the sequence $\<V_{\zeta}:{\zeta}<{\kappa}\>$
is decreasing and 
\begin{displaymath}
 \bigcap_{{\zeta}<{\kappa}}V_{\zeta}\cap Z=\empt
\end{displaymath}
by the Claim.

Since $X$ is ${\kappa}$-c.c. there is ${\xi}<{\kappa}$
such that $\overline{V_{\zeta}}=\overline{V_{\xi}}$ for all
${\xi}<{\zeta}<{\kappa}$. 

We can assume that ${\xi}=0$.  Let  
\begin{displaymath}
T_{{\zeta}}=\left\{
\begin{array}{ll}
V_0\setminus Z&\text{if ${\zeta}=0$,}\\ \\
((\bigcap_{{\xi}<{\zeta}}V_{\xi})\setminus V_{\zeta})\cap Z&\text{if ${\zeta}>0$.} 
\end{array}
\right.
\end{displaymath}
Then 
\begin{displaymath}
\bigcup_{{\xi}<{\zeta}}T_{\zeta}\supset V_{\xi}\cap Z\subs^{dense}V.  
\end{displaymath}
thus the partition 
$\{T_{\zeta}:{\zeta}<{\kappa}\}$ witnesses that 
$V$ is \wreso{{\kappa}}.
\end{proof}

\begin{proof}[Proof of Theorem \ref{tm:ccc-spaces-below-ww}]
Let $\mc Y=\{Y\in \tau_X: |Y|=\Delta(Y)\}$.

Then $\bigcup \mc Y$ is dense in $X$, and every open subset of every $Y \in \mc Y$ is also in $\mc Y$.  
Thus by lemma \ref{lm:open-subspace} it is enough to prove that a c.c.c. space $Y$ 
with $\oo\le |Y|=\Delta(Y) <\omega_{\omega}$ is \wreso{\oo}.

Let $Y \in \mc Y$ such that $\omega_{n} = |Y|$.  Clearly, $Y$ is \wreso{\omega_{n}} as $|Y|=\Delta(Y)=\omega_{n}$.  Since $Y$ is c.c.c. then $Y$ is $\omega_{n-1}$-c.c..  By theorem \ref{tm:stepping-down}, $Y$ is \wreso{\omega_{n-1}}.  By continually applying theorem \ref{tm:stepping-down} we conclude that $Y$ is \wreso{\oo}.
\end{proof}

\begin{problem}
Is it true that every crowded c.c.c  space with $\Delta(X)\ge \oo$
 is \wreso{\oo}?
\end{problem}

\section{Spaces which are not \wreso{\oo}.}

If $X$ is a topological space, and $\mc D\subs \mc P(D)$,
write
\begin{displaymath}
\clo{\mc D} =\{\overline D:D\in \mc D\}
\end{displaymath}

\begin{lemma}\label{lm:clo-point-countable}
Let $X$ be a topological space. Assume that $\clo {\mc D}$
is point-countable for each point-countable family $\mc D\subs \mc P(X)$.
Then $X$ is not contain any \wreso{\oo} subspace $Y$.
\end{lemma}

\begin{proof}
Assume that  $\{Y_{\zeta}:{\zeta}<\oo\}$ is a partition of $Y$.
Let $D_{\xi}=\bigcup\{Y_{\zeta}:{\xi}<{\zeta}\}$ for ${\xi}<\oo$.
Then the family $\mc D=\{D_{\xi}:{\xi}<\oo\}$ is point-countable.
So $\clo{\mc D}$ is also point-countable. So ${D_{\xi}}$ is not dense in $Y$ 
for all but countably many ${\xi}$. So the partition 
$\{Y_{\zeta}:{\zeta}<\oo\}$ does not witness that $Y$ 
is \wreso{\oo}.
\end{proof}

To prove Theorems \ref{tm:example-of-size-measurable} and \ref{tm:example-of-size-ww}
we should recall some definitions and results from 
\cite{JuSoSz} and \cite{JuMa}.

\begin{definition}[{\cite[Definition 3.1]{JuSoSz}}]
Let ${\kappa}$ be an infinite cardinal, and let $\mc F$ be a filter on
${\kappa}$. Let $T$ be the tree ${\kappa}^{<{\omega}}$. 
A topology $\tau_{\mc F}$ is defined on $T$ by
\begin{displaymath}
 \tau_{\mc F}=\big\{V\subs T: 
 \forall t\in V\{{\alpha}\in {\kappa}:t^\frown {\alpha}\in V  \}\in \mc F\big\},
\end{displaymath}
and the space $\<T,\tau_{\mc F}\>$ is denoted by $X(\mc F)$.
\end{definition}

\begin{proof}[Proof of Theorem \ref{tm:example-of-size-measurable}]
Let $\mc U$ be a ${\kappa}$-complete non-principal ultrafilter on ${\kappa}$.

The space $X=X(\mc U)$ is monotonically normal by \cite[Theorem 3.1]{JuSoSz}.

An ultrafilter $\mc U$  is {\em ${\lambda}$-descendingly complete} 
if $\bigcap\{U_{\zeta}:{\zeta}<{\lambda}\}\ne \empt$ for each decreasing
sequence $\{U_{\zeta}:{\zeta}<{\lambda}\}\subs \mc U$.

A  $\sigma$-complete ultrafilter is clearly 
${\omega}$-descendingly-complete. 
In the proof of \cite[Theorem 3.5]{JuSoSz}
the authors prove  Lemma 3.6 which claims that  
$\clo {\mc D}$
is point-countable for each point-countable family $\mc D\subs \mc P(X(\mc F))$
provided that $\mc F$ is a ${\omega}$-descendingly complete ultrafilter.
So $\clo {\mc D}$
is point-countable for each point-countable family $\mc D\subs \mc P(X)$,
and so $X$ is  not \wreso{\oo} by Lemma \ref{lm:clo-point-countable}.
\end{proof}

Instead of Theorem \ref{tm:example-of-size-ww}
we prove the following theorem which is a  slight improvement of
\cite[Theorem 5]{JuMa}.

\begin{theorem}\label{tm:example-of-size-ww2}
If it is consistent that there is a measurable  cardinal, 
then it is also consistent that  there is an 
${\omega}$-resolvable monotonically normal space $X$ with $|X|=\Delta(X)={\omega}_{\omega}$
such that if a family $\mc D\subs \mc P(X)$ is point-countable,
then the family  $\clo{\mc D}= \{\overline{D}:D\in \mc D\}$
is also point countable. Hence $X$ does not contain 
any \wreso{{{\omega}_1}} subspace.
\end{theorem}

\begin{proof}
In  \cite[page 665]{JuMa} the authors
write that ''{\em starting from
one measurable, Woodin (\cite{W}) constructed a model in which $\aleph_{\omega}$ carries an  $\oo$-
descendingly complete uniform ultrafilter. Woodin’s model $V_1$ can be embedded
into a bigger ZFC model $V_2$ so that the pair of models $(V1,V_2)$ with 
${\kappa} = \aleph_{\omega}$ 
satisfies the two models situation}'', i.e.

\begin{enumerate}[(1)]
 \item $\oo^{V_1}=\oo^{V_2}$,
\item there is a countable subset $A$ of 
${\omega}_{\omega}$  in $V_2$ such that no $B \in  V_1$ of
cardinality $<{\omega}_{\omega}$ covers $A$;
\item 
for the filter $\mc G$ on ${\omega}_{\omega}$ defined in 
$V_2$ by $B \in \mc G$ iff $A − B$ is finite, we have
$\mc G \cap V_1 \in V_1$. 
\end{enumerate}
(the ``two model situation'' is defined in   \cite[Theorem 4.5]{JuMa}).

\medskip
Let $\mc F=\mc G\cap V_1$ and consider the space $X=X(\mc F)$.
As it was observed in \cite{JuSoSz}, spaces obtained as  $X(\mc H)$ 
from some filter $\mc H$
are monotonically normal and ${\omega}$-resolvable.

In \cite[Theorem 4.1]{JuMa}
Juhász and Magidor showed that the space $X(\mc F)$ is actually
hereditarily ${\omega}_1$-irresolvable. They proved the following lemma:

\medskip 
 
\noindent{\bf Lemma 4.2 from \cite{JuMa}.}
 For any $D \subs X(F)$ and $t {\in} \overline{D}$ there is a finite sequence $s$ of
members of $A$ such that $t^\frown s {\in} D$.

\medskip

Using this lemma we show that  $\clo {\mc D}$
is point-countable for each point-countable family $\mc D\subs \mc P(X)$,
and so  $X$ is not \wreso{\oo} by Lemma \ref{lm:clo-point-countable}.

Indeed,  let $\mc D\subs \mc P(X)$ be an uncountable   family such that 
$t\in \bigcap_{D\in \mc D}\overline{D}$.
Then, by \cite[Lemma 4.3]{JuMa}, for each $D\in \mc D$ we can pick 
a finite sequence $s_D$ of
members of $A$ such that $t^\frown s_D {\in} D$.
Since there are only countable many finite sequences of elements of
$A$ there is $s$ such that 
$s_D=s$ for uncountably many $D\in \mc D$.
Then $t^\frown s$ is in uncountably many elements of $\mc D$,
so $\mc D$ is not point-countable.

So we proved that no subspace of $X$ is  \wreso{\oo}.
\end{proof}

\section{${\omega}$-resolvability after adding a single Cohen reals}

Before proving Theorem \ref{tm:w-not-enough}
we need some preparation. 

The notion of almost resolvability was introduced by 
Bolstein (\cite{Bo}) in 1973:  a topological space is {\em almost-resolvable} if
it is a countable union of sets with empty interiors.
The notion of \wresob{{\omega}} was first considered in \cite{TaVi}
under the name {almost-${\omega}$-resolvability}.

Clearly  almost ${\omega}$-resolvable (i.e. \wreso{{\omega}}) spaces are almost resolvable.

\begin{lemma}\label{lm:cohen-approx}
Let $X$ be a crowded topological space.\\
(1) If $X$ is \wreso{{\omega}}, then $X$ is ${\omega}$-resolvable in 
$V^{\Fn({\omega},2)}$.\\
(2)  If $X$ is resolvable in $V^{\Fn({\omega},2)}$, then 
$X$ is almost-resolvable.
\end{lemma}

\begin{proof}[Proof of Lemma \ref{lm:cohen-approx}]
(1) Assume that the function $f:X\to {\omega}$ witnesses
the \wresob{{\omega}} of $X$.

If  $\mc G$ is the  $V$-generic filter in $\Fn({\omega},{\omega})$, and
$g=\bigcup \mc G$, then the function $h=g\circ f$ witnesses
that $X$ is ${\omega}$-resolvable.

We need to show that 
$\{y\in X: (g\circ f)(y)=n\}$ is dense in $X$

Indeed, let $p\in \Fn({\omega},{\omega})$, $\empt\ne U\in \tau_X$.
Since $f:X\to {\omega}$ witnesses
the \wresob{{\omega}} of $X$ 
there is $y\in U$ such that 
\begin{displaymath}
f(y)> \max \dom(p). 
\end{displaymath}
Let 
\begin{displaymath}
  q= p\cup\{\<f(y),n\>\}.
\end{displaymath}
Then $q\le p$ and 
\begin{displaymath}
 g\Vdash (g\circ f)(y)=n.
\end{displaymath}

So we proved that $g\circ f$ witnesses that $X$ is ${\omega}$-resolvable in 
the generic extension.

\medskip
\noindent (2) 
Assume 
\begin{displaymath}
 V^{\Fn({\omega},2)}\models\text{``$X$ has a partition $\{D_0,D_1\}$ into dense subsets.''}
\end{displaymath}
For all $p\in \Fn({\omega},2)$ and $i<2$
let
\begin{displaymath}
 D^p_{i}=\{x\in X: p\Vdash x\in \dot{D_i}\}.
\end{displaymath}
Then $X=\bigcup\{D^p_i:p\in \Fn({\omega},2), i<2\}$,
and we claim that $\inte D^p_i=\empt$ for each $p\in \Fn({\omega},2)$, 
and $i<2$.

Indeed,   fix $p$ and $i$ and let $U$ be an arbitrary  non-empty open subset. 
Then $p\Vdash U\cap \dot{D_{1-i}}\ne \empt$, so there is $q\le p$
and $y\in U$ such that  $q\Vdash y\in \dot{D_{1-i}}$.
Then $q\Vdash y\notin \dot{D_{i}}$, so 
$p\not\Vdash y\in \dot{D_{i}}$, and so $y\notin D^p_i$.
Thus $U\not\subs D^p_i$.   Since $U$ was arbitrary, we proved
$\inte D^p_i=\empt$.
\end{proof}

After this preparation we can prove Theorem \ref{tm:w-not-enough}.

\begin{proof}[Proof of Theorem \ref{tm:w-not-enough}]
(1)
 Kunen \cite{Ku} proved that it 
is consistent, modulo a measurable cardinal,
that there is a maximal independent family 
$\mc A \subs  \mc P(\oo)$
which is also $\sigma$-independent.

In  \cite[Theorems 3.1 and 3.2]{KuSzyTa} the authors proved that 
if there is a maximal independent family
$\mc A \subs  \mc P(\oo)$
which is also $\sigma$-independent,
then  there is a Baire space $X$
 with $|X|=\Delta(X)=\oo$ such that every open subspace of $X$
  is irresolvable, i.e. the  space  $X$ is {\em OHI}. 

It is well-known that a crowded  OHI Baire space $X$ is not almost resolvable:
if $X=\bigcup_{n\in {\omega}}X_n$, then $\inte X_n\ne \empt$
for some $n\in {\omega}$.  

Indeed, if $\inte X_n=\empt$, then $X\setminus X_n$ is dense, so 
$U_n=\inte (X\setminus X_n)$ is dense in $X$ because 
every open subset of $X$ is irresolvable.
Thus $\bigcap_{n\in {\omega}}U_n\ne\empt$ because $X$ is Baire.
However 
\begin{displaymath}
\bigcap_{n\in {\omega}}U_n\subs \bigcap_{n\in {\omega}}(X\setm X_n)=
X\setminus \bigcup_{n\in {\omega}}X_n=\empt,
\end{displaymath}
which is a contradiction.

Thus $X$ is not almost resolvable, so it is not 
${\omega}$-resolvable in the model $V^{\Fn({\omega},2)}$
by Lemma \ref{lm:cohen-approx}(2).
\medskip

\noindent (2)
In \cite{KuTa} the authors  proved that if $V=L$, then there are no
crowded Baire irresolvable spaces. 
Hence, by \cite{TaVi}, if $V=L$, 
 then every crowded  space $X$ 
is almost-${\omega}$-resolvable (i.e. \wreso{{\omega}}).

So these spaces are   ${\omega}$-resolvable in the model $V^{\Fn({\omega},2)}$
by Lemma \ref{lm:cohen-approx}(1).
\end{proof}

\begin{proof}[Proof of Theorem \ref{tm:w-not-enough}(3)]
Let $X$ be a crowded c.c.c space.

We can assume that $|X|=|\Delta(X)$.

By induction we define a strictly decreasing sequence of cardinals:
\begin{displaymath}
 {\kappa}_0, {\kappa}_1,\dots,{\kappa}_n\dots
\end{displaymath}
as follows.
\begin{enumerate}[(i)]
 \item ${\kappa}_0=\Delta(X)$,
 \item if ${\kappa}_i$ is singular, then ${\kappa}_{i+1}=\cf({\kappa}_i)$,
 \item if ${\kappa}_i>{\omega}$ is regular, then ${\kappa}_i={\lambda}^+$
 (because $|X|$ is below the first weakly inaccessible cardinal,) and let
 ${\kappa}_{i+1}={\lambda}$,
 \item if ${\kappa}_i={\omega}$ or ${\kappa}_i={\omega}_1$, then we stop.
\end{enumerate}
Assume that the construction stopped in the $n$th step.

Then we can prove, by finite induction, then 
$X$ is \wreso{{\kappa}_i} for all $i\le n$ by theorem  \ref{tm:stepping-down}.
Thus $X$ is \wreso{{\omega}}  or \wreso{{\omega_1}}, and so either $X$
is ${\omega}$-resolvable in $V^{\Fn({\omega},2)}$    by 
by Lemma \ref{lm:cohen-approx}(1), or 
$X$
is ${\omega}_1$-resolvable in $V^{\Fn({\omega}_1,2)}$    by 
Thereon \ref{tm:kappa-res-in-ccc}.
\end{proof}

\begin{problem}[{\cite[Questions 5.2.]{TaVi}}]
Are almost resolvability and almost-${\omega}$-resolvability equivalent in the class
of irresolvable spaces? 
\end{problem}

\begin{problem}
Is there, in ZFC,  a crowded topological space $X$ which is irresolvable 
in the Cohen generic extension $V^{\Fn({\omega},2)}$.
\end{problem}


\begin{thebibliography}{XX}


\bibitem{AIT}
J. Angoa; M. Ibarra; Angel Tamariz-Mascarúa
{\em On $\omega$ -resolvable and almost-${\omega}$ -resolvable spaces}
Commentationes Mathematicae Universitatis Carolinae, Vol. 49 (2008), No. 3, 485--508





\bibitem{Bo}
Richard Bolstein
{\em Sets of Points of Discontinuity}
Proceedings of the
American Mathematical society
38.1 (1973), 193-197


\bibitem{DoAl}
A. Dorantes-Aldama, {\em Baire irresolvable spaces with countable Souslin number},
Topology Appl. 188, (2015) 16–26.




\bibitem{He}  Hewitt, E. {\em A problem of set theoretic topology},
Duke Math. J. 10 (1943) 309-333.

\bibitem{JuMa}
István Juhász, Menachem Magidor, {\em 
On the Maximal Resolvability of Monotonically Normal Spaces}
Israel Journal Of Mathematics 192 (2012), 637–666.



\bibitem{JuSoSz} Juhász, István; Soukup, Lajos; Szentmiklóssy, Zoltán 
 {\em Resolvability and monotone normality}
Israel
 Journal of Mathematics 166 (2008), 1–16.

\bibitem{Ku}
 Kunen, Kenneth
{\em Maximal $\sigma$-independent families.}
Fund. Math. 117 (1983), no. 1, 75–80. 
 
 
 

\bibitem{KuPr}
Kenneth Kunen, Karel Prikry,
{On descendingly incomplete ultrafilters},
The journal of symbolic logic
volume 36.4 (1971), pp 650--652.



\bibitem{KuSzyTa}  K. Kunen, A. Szymanski, F. Tall
{\em Naire irresolvable spaces and ideal theory}
Ann. Math Syleziana 2(14) (1986) 98.107

\bibitem{KuTa}
K. Kunen, F. Tall, {\em On the consistency of the non-existence of 
Baire irresolvable spaces}, Manuscript privately circulated,
Topology Atlas, http://at.yorku.ca/v/a/a/a/27.htm, 1998.





\bibitem{Pa} 
O. Pavlov, Problems on (ir)resolvability, in: Open Problems in Topology II, Elsevier, 2007, pp. 51–59.




\bibitem{Recent}
{\em Recent Progress in General Topology, III} ed. K.P. Hart, Jan van Mill, P Simon
Springer Science \& Business Media, 2013, pp 903.




\bibitem{TaVi}
A. Tamariz-Mascarúa, H. Villegas-Rodríguez
{em  Spaces of continuous functions, box
products and almost-$\omega$-resolvable spaces}
Comment.Math.Univ.Carolin. 43,4 (2002)687–705


\bibitem{Ul}
 S. Ulam, Zur Masstheorie in der allgemeinen Mengenlehre, Fundamenta Mathematicae,
vol. 16 (1930), pp. 140-150.

\bibitem{W} 
W. H. Woodin, {\em Descendingly complete ultrafilter on $\aleph_{\omega}$}, 
Personal communication.



\end{thebibliography}
\end{document}